\documentclass[11pt]{amsart}
\usepackage{amssymb}
\usepackage[margin=1.15in]{geometry}
\newtheorem{theorem}{Theorem}[section]
\newtheorem{corollary}[theorem]{Corollary}
\newtheorem{lemma}[theorem]{Lemma}

\theoremstyle{definition}
\newtheorem{definition}[theorem]{Definition}
\newtheorem{example}[theorem]{Example}
\theoremstyle{remark}
\newtheorem{remark}[theorem]{Remark}

\title[Compact disjointness preserving operators on $C(K)$-modules]{Compact disjointness preserving operators on Banach $C(K)$-modules}
\author{Arkady Kitover}
\address{Community College of Philadelphia, 1700 Spring Garden St., Philadelphia, PA, USA}
\email{akitover@ccp.edu}
\author{Mehmet Orhon}
\address{University of New Hampshire, 105 Main Street, Durham, NH 03824, USA}
\email{mo@unh.edu}
\subjclass[2020]{Primary 47B07; Secondary 46B42, 47B60}
\keywords{Compact operator, disjointness preserving operator, Banach $C(K)$-module}

\begin{document}
\begin{abstract}
We show that some well known results concerning compact disjointness preserving operators on Banach lattices can be extended to the more general framework of finitely generated Banach $C(K)$-modules. Specifically, we establish a series representation for such operators analogous to the de Pagter--Wickstead and Arenson--Kitover theorems and provide some details about their spectra.
\end{abstract}
\maketitle

\noindent\emph{Note to the revised version.} This version corrects Theorem~2.1 of the published article (Positivity (2026) 30:82). Formula~(3) was stated incorrectly, and the implication (1)$\Rightarrow$(3) is false (Example~2.4). Theorem~2.1 is restated: conditions (2), (3) and (4) are equivalent, and (1) is strictly weaker. Complete proofs of (3)$\Rightarrow$(4) and of (2)$\Rightarrow$(3) are given. The proof of Corollary~3.4 is completed: the published proof produces an atom of $X(x)'$, and it is shown here how to obtain an atom of $X'$. The statements of Kamowitz's theorems (Theorems~1.3 and~3.1) are corrected. Theorem~3.2 is now stated and proved for compact operators, which is the case needed in Corollary~3.4; the proof of the published Theorem~3.2 for non-compact operators was incomplete (Remark~3.3).

\section{Introduction and preliminaries}
Compact lattice homomorphisms acting between two Banach lattices were characterized independently by B.~de Pagter (\cite{dP}) and A.W.~Wickstead (\cite{W1}). To state their result we recall the following definition.

\begin{definition}
Let $X$ be a vector lattice. An element $x\in X$ is called an \emph{atom} if the principal ideal $I_x$ generated by $x$ is one-dimensional.
\end{definition}

\begin{theorem}[de Pagter--Wickstead]
Let $X$, $Y$ be Banach lattices and $T:X\to Y$ be a compact lattice homomorphism. Then there are pairwise disjoint positive elements $y_1,y_2,\dots\in Y$ and pairwise disjoint positive atoms $F_1,F_2,\dots\in X'$ such that
\begin{equation}\label{1}
T=\sum_{n=1}^\infty F_n\otimes y_n,
\end{equation}
where the series in \eqref{1} converges in the operator norm.
\end{theorem}

H.~Kamowitz obtained (\cite{Ka}) the following characterization of compact weighted composition operators on $C(K)$.

\begin{theorem}[Kamowitz]
Let $K$ be a compact Hausdorff space and $C(K)$ be the space of all continuous complex-valued functions on $K$. Let $w\in C(K)$ and let $\varphi$ be a continuous map of $K$ into itself. The operator $T$, $Tf=w(f\circ\varphi)$, $f\in C(K)$, is compact if and only if every component of the set $\{k\in K:|w(k)|>0\}$ has a neighbourhood on which $\varphi$ is constant.
\end{theorem}

The above-mentioned results were extended to arbitrary compact disjointness preserving operators between Banach lattices by E.L.~Arenson and the first-named author in \cite{AK}.

\begin{theorem}[Arenson--Kitover]\label{t:AK}
Let $X$, $Y$ be Banach lattices and $T:X\to Y$ be a compact disjointness preserving operator. Then there are pairwise disjoint elements $y_1,y_2,\dots\in Y$ and pairwise disjoint atoms $F_1,F_2,\dots\in X'$ such that
\begin{equation}\label{2}
T=\sum_{n=1}^\infty F_n\otimes y_n,
\end{equation}
where the series in \eqref{2} converges in the operator norm.
\end{theorem}

\begin{remark}
Although Theorems~1.2 and~1.4 are formally similar in statement, extending Theorem~1.2 to all compact disjointness preserving operators requires some additional work (see \cite[Lemma~3]{AK}), as well as application of nontrivial results from \cite{A} and \cite{AVK}.
\end{remark}

The goal of the current paper is to obtain an analog of Theorem~1.4 for compact disjointness preserving operators acting between two Banach $C(K)$-modules.

For the reader's convenience we briefly recall some basic facts about Banach $C(K)$-modules (see \cite{AAK}).

A Banach space $X$ is called a \emph{Banach $C(K)$-module} if there is an isometry $m:C(K)\to L(X)$, where $L(X)$ is the space of all bounded linear operators on $X$.

The \emph{center} $Z(X)$ of a Banach $C(K)$-module $X$ is defined (see \cite[p.~33]{AAK}) as the closure of $C(K)$ in $L(X)$ in the strong operator topology. In the current paper we will always assume that $C(K)=Z(X)$.

For any $x\in X$ the subspace $X(x)$ of $X$, $X(x)=\operatorname{cl}_X\{fx:f\in C(K)\}$, can be endowed with the structure of a Banach lattice with the cone $X(x)_+=\operatorname{cl}_X\{fx:f\in C(K),\ f\ge0\}$. The space $X(x)$ is called a \emph{cyclic subspace} of $X$.

By analogy with Definition~1.1 we introduce the following definition.

\begin{definition}
Let $X$ be a Banach $C(K)$-module. An element $x\in X$ is called an \emph{atom} if the cyclic subspace $X(x)$ is one-dimensional.
\end{definition}

A Banach $C(K)$-module $X$ is called \emph{finitely generated} if there are $x_1,\dots,x_n\in X$ such that
\[
\operatorname{cl}\{f_1x_1+f_2x_2+\dots+f_nx_n:\ f_1,\dots,f_n\in C(K)\}=X .
\]

A Banach $C(K)$-module $X$ is called a \emph{Kaplansky module} if the compact space $K$ is extremally disconnected and for any $x\in X$ the set $\{f\in C(K):fx=0\}$ is a band in $C(K)$.\footnote{In \cite{AAK} Kaplansky modules are called order complete modules.}

Two elements $x_1,x_2$ of a Banach $C(K)$-module $X$ are called \emph{disjoint} ($x_1\,d\,x_2$) if $x_1,x_2\in X(x_1+x_2)$ and $x_1,x_2$ are disjoint elements of the Banach lattice $X(x_1+x_2)$.

Let $X$ be a Banach $C(K_1)$-module and $Y$ be a Banach $C(K_2)$-module. A linear operator $T:X\to Y$ is called \emph{disjointness preserving} if $x_1\,d\,x_2\Rightarrow Tx_1\,d\,Tx_2$. We will need the following theorem proved in \cite{KO6}.

\begin{theorem}
Let $X$ be a Banach $C(K_1)$-module and $Y$ be a Banach $C(K_2)$-module. Let $T:X\to Y$ be a linear operator. The following conditions are equivalent.
\begin{enumerate}
\item[(1)] $T$ is a disjointness preserving operator.
\item[(2)] For any $x\in X$ we have $T(X(x))\subset Y(Tx)$ and the restriction of $T$ on $X(x)$ is a disjointness preserving operator between the Banach lattices $X(x)$ and $Y(Tx)$.
\end{enumerate}
\end{theorem}

And we need as well the following (see \cite[Proposition~11.23]{AAK}).

\begin{theorem}\label{t:phi}
Let $X$ be a Banach $C(K_1)$-module and $Y$ be a Kaplansky $C(K_2)$-module. Let $T:X\to Y$ be a disjointness preserving operator. Let $E=\operatorname{supp}(TX)\subset K_2$. Then $E$ is clopen, and there is a continuous map $\varphi:E\to K_1$ such that $Tf=(f\circ\varphi)T$, $f\in C(K_1)$.
\end{theorem}

We shall also use the following facts about Kaplansky modules and second duals.
\begin{enumerate}
\item[(K1)] For every Banach $C(K)$-module $Y$, the second dual $Y''$ is a Kaplansky module over its center \cite[Proposition~11.31]{AAK}, applied to the conjugate module $Y'$.
\item[(K2)] The canonical embedding $j:Y\to Y''$ preserves disjointness \cite[Lemma~11.34 and its proof]{AAK}.
\item[(K3)] If $u,v\in Y$ and $ju$, $jv$ are disjoint in $Y''$, then $u$ and $v$ are disjoint in $Y$.

By the proof of \cite[Lemma~11.34]{AAK}, $u$ and $v$ are disjoint iff $\inf\{\|fu\|+\|(1-f)v\|:f\in C(K)\}=0$, and the same characterization holds in $Y''$ with $h\in Z(Y'')$ in place of $f$. Given $h\in Z(Y'')$, \cite[Theorem~9.15 and the proof of Proposition~11.36]{AAK} provides a bounded net $(f_\alpha)$ in $C(K)$ with $(f_\alpha ju,(1-f_\alpha)jv)\to(h\,ju,(1-h)jv)$ weak$^*$ in $Y''\oplus_1Y''$. By the Hahn--Banach theorem, the norm of a weak$^*$ limit of a convex subset of $Y\oplus_1Y$ is at least the distance of that set from $0$. Hence $\inf_f(\|fu\|+\|(1-f)v\|)\le\|h\,ju\|+\|(1-h)jv\|$, and the infimum over $h$ is $0$.
\item[(K4)] In a Kaplansky module, the support of every element is clopen \cite[11.19.1]{AAK}, the center contains $\chi_U$ for every clopen $U$ \cite[Proposition~11.20]{AAK}, and two elements are disjoint iff their supports are disjoint \cite[Corollary~11.22.3]{AAK}.
\end{enumerate}

\section{The main result}

\begin{theorem}\label{t:main}
Let $X$ be a finitely generated Banach $C(K_1)$-module with generators $x_1,\dots,x_p$, let $Y$ be a Banach $C(K_2)$-module, and let $T:X\to Y$ be a bounded linear disjointness preserving operator. Consider the following conditions.
\begin{enumerate}
\item[(1)] The operators $T:X(x_j)\to Y(Tx_j)$, $j=1,\dots,p$, are compact.
\item[(2)] For any $x\in X$ the operator $T:X(x)\to Y(Tx)$ is compact.
\item[(3)] The operator $T$ is compact.
\item[(4)] $T$ has the representation
\begin{equation}\label{3}
T=\sum_{n=1}^{\infty}\ \sum_{j=1}^{k(n)}H_{j,n}\otimes y_{j,n},
\end{equation}
where $k(n)\le p$ for every $n$, and there are pairwise distinct points $s_n\in K_1$ such that $fH_{j,n}=f(s_n)H_{j,n}$ for all $f\in C(K_1)$, $n\in\mathbb N$, $1\le j\le k(n)$. In particular, each $H_{j,n}$ is an atom in $X'$, and $F_n=H_{1,n}+\dots+H_{k(n),n}$ is an atom in $X'$ \cite[Corollary~7.7.1]{AAK}; atoms corresponding to different $n$ are disjoint. Moreover, the elements $y_{i,m}$ and $y_{j,n}$, where $m\ne n$, $1\le i\le k(m)$, $1\le j\le k(n)$, are disjoint, and the series in \eqref{3} converges in the operator norm.
\end{enumerate}
Then \textup{(2)}$\Leftrightarrow$\textup{(3)}$\Leftrightarrow$\textup{(4)}$\Rightarrow$\textup{(1)}. In general, (1) does not imply (3) (Example~\ref{e:counter}).
\end{theorem}

\begin{remark}\label{r:distinct}
In condition (4), atoms corresponding to different points are disjoint. If $fH=f(s)H$ and $fG=f(t)G$ with $s\ne t$, choose $f\in C(K_1)$ with $f(s)=1$ and $f(t)=0$. Then $H=f(H+G)$ and $G=(1-f)(H+G)$ lie in the cyclic subspace of $H+G$ in $X'$, which is spanned by $H$ and $G$, and they are disjoint there.
\end{remark}

\begin{remark}\label{r:rank}
The $n$-th term of \eqref{3} has rank $k(n)$ in general, and it cannot be replaced by the rank-one operator $F_n\otimes\sum_jy_{j,n}$. Let $K$ be a single point and $X=\mathbb C^2$ with the Euclidean norm, a Banach $C(K)$-module generated by two elements, with $Z(X)=C(K)$. Every cyclic subspace of $X$ is at most one-dimensional, so two elements of $X$ are disjoint only if one of them is zero. Hence every linear operator on $X$ is disjointness preserving and compact; in particular so is the identity, which has rank $2$. Applying the same argument to $X'=\mathbb C^2$, a series of rank-one terms $F_n\otimes v_n$ with pairwise disjoint atoms $F_n$ has at most one nonzero term, and therefore cannot represent the identity. In the form \eqref{3}, the identity is $H_1\otimes e_1+H_2\otimes e_2$, with $k(1)=2=p$.
\end{remark}

\begin{proof}[Proof of Theorem~\ref{t:main}]
The implications (4)$\Rightarrow$(3)$\Rightarrow$(2)$\Rightarrow$(1) are trivial. We prove (3)$\Rightarrow$(4).

Let $T$ be compact. The restrictions $T:X(x_j)\to Y(Tx_j)$ are compact. By Theorem~\ref{t:AK}, applied with values in $Y$ (so that all $z_{j,m}$ below lie in $Y$),
\begin{equation}\label{4}
Tx=\sum_mG_{j,m}(x)z_{j,m}\qquad(x\in X(x_j),\ 1\le j\le p),
\end{equation}
where the $G_{j,m}$ are pairwise disjoint atoms of $X(x_j)'$, the $z_{j,m}$ are pairwise disjoint in $Y(Tx_j)$, and the series converges in operator norm. Because the $z_{j,m}$ are disjoint, the convergence is unconditional.

\emph{The points $s_{j,m}$.} For $f\in C(K_1)$ we have $|f'G_{j,m}|\le\|f\|\,|G_{j,m}|$, where $f'$ is the adjoint of multiplication by $f$ on $X(x_j)$. Since $G_{j,m}$ is an atom, $f'G_{j,m}=\lambda(f)G_{j,m}$ with $\lambda(f)\in\mathbb C$. Clearly $\lambda$ is a bounded multiplicative functional on $C(K_1)$, so there is $s_{j,m}\in K_1$ with $\lambda(f)=f(s_{j,m})$, that is, $G_{j,m}(fx)=f(s_{j,m})G_{j,m}(x)$ for $x\in X(x_j)$. Moreover $c_{j,m}:=G_{j,m}(x_j)\ne0$, since otherwise $G_{j,m}(fx_j)=f(s_{j,m})c_{j,m}=0$ for all $f$. For fixed $j$ the points $s_{j,m}$ are pairwise distinct: two such functionals with the same point are both determined by their value at $x_j$, so they are proportional and cannot be disjoint.

\emph{Passing to $Y''$.} We regard $T$ as an operator into $Y''$; by (K1) and (K2), $Y''$ is a Kaplansky module and $T:X\to Y''$ is disjointness preserving. (If $Y$ itself is a Kaplansky module, this step is unnecessary.) Let $\varphi:E\to K_1$ be the map of Theorem~\ref{t:phi} for $T:X\to Y''$, so that $T(fx)=(f\circ\varphi)Tx$, with $f\circ\varphi$ extended by $0$ outside the clopen set $E$.

Let $Q=\chi_{\operatorname{supp}z_{j,m}}$, a band projection by (K4). Since $Qz_{j,l}=0$ for $l\ne m$ (by (K2) and (K4)),
\[
QT(fx_j)=f(s_{j,m})c_{j,m}z_{j,m}
\quad\text{and}\quad
QT(fx_j)=(f\circ\varphi)QTx_j=(f\circ\varphi)c_{j,m}z_{j,m}.
\]
Hence $f\circ\varphi=f(s_{j,m})$ on $\operatorname{supp}z_{j,m}$ for all $f$, that is, $\varphi(\operatorname{supp}z_{j,m})=\{s_{j,m}\}$. Consequently, elements $z_{i,m}$, $z_{j,l}$ with $s_{i,m}\ne s_{j,l}$ have disjoint supports.

\emph{The blocks.} Let $(s_n)$ be an enumeration of the distinct points among the $s_{j,m}$. For each $n$ and $j$ there is at most one $m$ with $s_{j,m}=s_n$. Let $P_n=\chi_{U_n}$, where $U_n$ is the (clopen) union of the supports of the $z_{j,m}$ with $s_{j,m}=s_n$, and put $T_n=P_nT$. The sets $U_n$ are pairwise disjoint.

For $x\in X(x_j)$ we have $T_nx=G_{j,m}(x)z_{j,m}$ if $s_{j,m}=s_n$, and $T_nx=0$ if no such $m$ exists. Let $X_0$ be the linear span of the sets $C(K_1)x_j$, a dense subspace of $X$. Then $T_n(X_0)$ is contained in the span of at most $p$ elements $z_{j,m}$, which lie in $Y$. This span is finite-dimensional and hence closed, so $T_n(X)\subseteq Y$ and $\operatorname{rank}T_n=:k(n)\le p$. Every functional $H=T_n'y'$ satisfies $H(fx)=f(s_n)H(x)$ on $X_0$, hence on $X$. Choosing a basis $y_{1,n},\dots,y_{k(n),n}$ of $T_n(X)$ among the $z_{j,m}$, we get
\[
T_n=\sum_{j=1}^{k(n)}H_{j,n}\otimes y_{j,n},\qquad fH_{j,n}=f(s_n)H_{j,n}.
\]
Elements $y_{i,m}$, $y_{j,n}$ with $m\ne n$ have disjoint supports, hence are disjoint in $Y''$ by (K4) and therefore in $Y$ by (K3).

Finally, for $x\in X_0$, $Tx=\sum_nT_nx$ with unconditional convergence in norm, since on each $X(x_j)$ this series is a rearrangement of the series \eqref{4}.

\emph{Convergence in operator norm.} For a finite set $A\subseteq\mathbb N$ put $P_A=\sum_{n\in A}P_n=\chi_{\bigcup_{n\in A}U_n}$. Then $\|P_A\|\le1$ and $P_AT=\sum_{n\in A}T_n$. We show that $\|T-\sum_{n\le N}T_n\|\to0$.

\emph{Step 1.} Let $x\in X$ and $\eta>0$. Choose $x_0\in X_0$ with $\|T\|\,\|x-x_0\|<\eta$, and then $N_0$ such that $\|\sum_{n\in A}T_nx_0\|<\eta$ for every finite $A\subseteq(N_0,\infty)$. For such $A$,
\[
\|P_ATx\|\le\|P_AT(x-x_0)\|+\|P_ATx_0\|<2\eta .
\]

\emph{Step 2.} Suppose that $\|T-\sum_{n\le N}T_n\|$ does not tend to $0$. Then there are $\delta>0$, finite sets $A_1<A_2<\cdots$ (that is, $\max A_i<\min A_{i+1}$), and $x_i\in X_0$ with $\|x_i\|\le1$ such that $\|P_{A_i}Tx_i\|\ge\delta$.

Indeed, take $N$ with $\|T-\sum_{n\le N}T_n\|\ge2\delta$. Since the unit ball of $X_0$ is dense in that of $X$, there is $x\in X_0$, $\|x\|\le1$, with $\|\sum_{n>N}T_nx\|>\delta$. Since this series converges, $\|\sum_{N<n\le M}T_nx\|\ge\delta$ for some $M$. Then repeat with some $N'\ge M$.

\emph{Step 3.} Since $T$ is compact, passing to a subsequence we may assume $Tx_i\to u$ in $Y$. Choose $x\in X$ with $\|Tx-u\|<\delta/4$. Then
\[
\|P_{A_i}Tx_i\|\le\|Tx_i-u\|+\|u-Tx\|+\|P_{A_i}Tx\| .
\]
The first term tends to $0$, the second is less than $\delta/4$, and by Step~1 the third tends to $0$. So $\|P_{A_i}Tx_i\|<\delta/2$ for large $i$, a contradiction.

Hence the series \eqref{3} converges to $T$ in operator norm, and (4) holds.
\end{proof}

The following example shows that condition (1) does not imply (3), so the generators alone do not determine compactness. The proof of (2)$\Rightarrow$(3) is given after the example.

\begin{example}\label{e:counter}
Let $K=\mathbb N\cup\{\infty\}$ be the one-point compactification of $\mathbb N$, and let
\[
X=\{x=(x^{(1)},x^{(2)})\in C(K;\mathbb C^2):\ x^{(2)}(\infty)=0\},
\]
with the norm $\|x\|=\sup_{k\in K}\max(|x^{(1)}(k)|,|x^{(2)}(k)|)$ and the pointwise action of $C(K)$. Since $(1,0)\in X$, the map $C(K)\to L(X)$ is an isometry. Let $Y=C(K)$, as a Banach $C(K)$-module.

\emph{The centers.} If $(g_\alpha)\subseteq C(K)$ and $g_\alpha x\to Sx$ for every $x\in X$, then $(g_\alpha,0)=g_\alpha(1,0)$ converges in norm. Hence $g_\alpha\to h$ uniformly for some $h\in C(K)$, and $S$ is multiplication by $h$. Thus $Z(X)=C(K)$. The same argument, with $\mathbf 1$ in place of $(1,0)$, gives $Z(Y)=C(K)$.

\emph{Generators.} Let $\varepsilon_k\in(0,1]$ with $\varepsilon_k\to0$, and put $\varepsilon(\infty)=0$. Let $x_1=(1,0)$ and $x_2=(1,\varepsilon)$. Then $fx_1+gx_2=(f+g,\varepsilon g)$. Let $z\in X$ with $z^{(2)}$ supported in a finite subset of $\mathbb N$. Take $g=z^{(2)}/\varepsilon$ on the support of $z^{(2)}$ and $g=0$ elsewhere, and $f=z^{(1)}-g$. Then $g,f\in C(K)$ and $fx_1+gx_2=z$. Such $z$ are dense in $X$, because $z^{(2)}(\infty)=0$. So $x_1,x_2$ generate $X$.

\emph{Disjointness.} For $z\in X$, every element of $X(z)$ has the form $k\mapsto\psi(k)z(k)$, and $\|\psi z\|=\sup_k|\psi(k)|\,\|z(k)\|$. The positive cone of $X(z)$ consists of the elements with $\psi\ge0$, and the lattice operations of $X(z)$ act pointwise on $\psi$ (if $f_nz\to\psi z$ with $\psi\ge0$, then $\|(|f_n|-\psi)z\|\le\|(f_n-\psi)z\|$). Hence, if $x,y\in X$ are disjoint, then for every $k$ either $x(k)=0$ or $y(k)=0$. The same holds in $Y$. Conversely, two elements $u,v\in Y$ with disjoint cozero sets are disjoint: $u=\lim_n\min(1,n|u|)(u+v)\in Y(u+v)$, similarly for $v$, and they are disjoint in the lattice $Y(u+v)$.

\emph{The operator.} Define
\[
T:X\to Y,\qquad Tx=x^{(2)} .
\]
$T$ is bounded, and it is disjointness preserving by the preceding paragraph.
\begin{itemize}
\item $Tx_1=0$, so $T$ vanishes on $X(x_1)$.
\item $X(x_2)=\{gx_2:g\in C(K)\}$ is closed, since $\|gx_2\|=\sup|g|$ (because $\varepsilon\le1$). Also $T(gx_2)=\varepsilon g$. The operator $g\mapsto\varepsilon g$ is a norm limit of finite rank operators, so $T:X(x_2)\to Y(Tx_2)$ is compact.
\item The map $h\mapsto(0,h)$ is an isometry of $\{h\in C(K):h(\infty)=0\}\cong c_0$ into $X$, and $T(0,h)=h$. Hence $T$ is not compact.
\end{itemize}
Thus (1) holds and (3) fails. Condition (2) fails as well: if $h(k)\ne0$ for all $k\in\mathbb N$, then $T$ is an isometry of the infinite-dimensional space $X((0,h))$.

In this example, an element $x=\sum_jf_jx_j$ of norm one may require summands $f_jx_j$ of very large norm. For instance, if $h\in C(K)$ is supported in a finite subset of $\mathbb N$ and $g=h/\varepsilon$ there and $0$ elsewhere, then $(0,h)=gx_2-gx_1$ and $\|gx_2\|=\max_k|h(k)|/\varepsilon_k$.
\end{example}

\subsection*{Proof of (2)$\Rightarrow$(3)}
Assume (2). The construction in the proof of (3)$\Rightarrow$(4), up to and including the identity $Tx=\sum_nT_nx$ for $x\in X_0$, uses only the boundedness of $T$ and the compactness of the restrictions $T:X(x_j)\to Y(Tx_j)$. So we have the points $s_n$, the band projections $P_n$ of $Y''$, and the blocks $T_n=P_nT$. For finite $A\subseteq\mathbb N$ put $S_A=\{s_n:n\in A\}$. Then $P_AT=\sum_{n\in A}T_n\in L(X,Y)$, and
\begin{equation}\label{e:star}
P_AT(fy)=\sum_{n\in A}f(s_n)T_ny=P_AT(gy)\qquad\text{whenever } f=g \text{ on } S_A .
\end{equation}
For $x\in X_0$, $(T-\sum_{n\le N}T_n)x=\lim_MP_{(N,M]}Tx$. Hence $T$ is the norm limit of $\sum_{n\le N}T_n$, and so compact, as soon as
\[
\tau_N:=\sup\{\|P_AT\|:\ A\subseteq(N,\infty)\ \text{finite}\}\to0 .
\]
Note that $\tau_N$ does not increase with $N$.

\begin{lemma}\label{l:hump}
Assume (2). There do not exist $\varepsilon>0$, finite sets $A_1,A_2,\dots\subseteq\mathbb N$, pairwise disjoint open sets $U_k\supseteq S_{A_k}$, and $y_k\in X$ with $\|y_k\|\le1$ and $\|P_{A_k}Ty_k\|\ge\varepsilon$.
\end{lemma}

\begin{proof}
Suppose such data exist. Choose $\varphi_k,\psi_k\in C(K_1)$ with values in $[0,1]$, $\varphi_k=1$ on $S_{A_k}$, $\psi_k=1$ on $\operatorname{supp}\varphi_k$, and $\operatorname{supp}\psi_k\subseteq U_k$. Put $x=\sum_k2^{-k}\varphi_ky_k$ and $u_k=2^k\psi_kx=\varphi_ky_k$. Then $u_k\in X(x)$ and $\|u_k\|\le1$. Since $\psi_k\ge0$ and $\psi_k\psi_l=0$ for $k\ne l$, the $u_k$ are pairwise disjoint in the Banach lattice $X(x)$. By Theorem~1.7 the $Tu_k$ are pairwise disjoint in the Banach lattice $Y(Tx)$. By \eqref{e:star}, $\|Tu_k\|\ge\|P_{A_k}Tu_k\|=\|P_{A_k}Ty_k\|\ge\varepsilon$. For disjoint $v,w$ in a Banach lattice, $\|v\|\le\|v-w\|$, so no subsequence of $(Tu_k)$ converges. This contradicts the compactness of $T:X(x)\to Y(Tx)$.
\end{proof}

\begin{lemma}\label{l:local}
Assume (2). For every $s\in K_1$ and $\varepsilon>0$ there are an open $V\ni s$ and $N\in\mathbb N$ such that $\|P_AT\|\le\varepsilon$ for every finite $A\subseteq(N,\infty)$ with $S_A\subseteq V$.
\end{lemma}

\begin{proof}
Suppose this fails for $s$ and $\varepsilon$. Choose $N_0$ such that $s_n\ne s$ for $n>N_0$, and put $V_1=K_1$. We choose inductively finite sets $A_k\subseteq(N_0,\infty)$ with $S_{A_k}\subseteq V_k$, elements $y_k\in X_0$ with $\|y_k\|\le1$ and $\|P_{A_k}Ty_k\|>\varepsilon$, open sets $U_k$ with $S_{A_k}\subseteq U_k\subseteq V_k$ and $s\notin\overline{U_k}$, and open sets $V_{k+1}\ni s$ with $\overline{V_{k+1}}\cap\overline{U_j}=\emptyset$ for $j\le k$. This is possible because the claim fails for $V_k$, because $X_0$ is dense, and because $s\notin S_{A_k}$. The $U_k$ are pairwise disjoint, which contradicts Lemma~\ref{l:hump}.
\end{proof}

\begin{proof}[Proof of (2)$\Rightarrow$(3)]
Suppose that $T$ is not compact. Then $\tau_N\ge2\delta>0$ for all $N$, and, as in Step~2 of the proof of (3)$\Rightarrow$(4), for every $N$ there are a finite $A\subseteq(N,\infty)$ and $y\in X_0$, $\|y\|\le1$, with $\|P_ATy\|\ge\delta$. Put $\varepsilon_k=\delta2^{-k-2}$. We construct inductively finite sets $A'_k\subseteq A_k$, elements $y_k$, open sets $U_k$ and $W_k$, and integers $N_k$, as follows.

Choose $A_k\subseteq(\max_{j<k}N_j,\infty)$ (with $\max\emptyset=0$) and $y_k$ with $\|y_k\|\le1$ and $\|P_{A_k}Ty_k\|\ge\delta$. Let
\[
A'_k=\{n\in A_k:\ s_n\notin W_1\cup\dots\cup W_{k-1}\}.
\]
For each $s\in S_{A'_k}$, Lemma~\ref{l:local} with $\varepsilon_k/|S_{A'_k}|$ in place of $\varepsilon$ gives an open $V_s\ni s$ and $N_s$. Put $W_k=\bigcup_{s\in S_{A'_k}}V_s$ and $N_k=\max_sN_s$. Splitting a finite set $B\subseteq(N_k,\infty)$ with $S_B\subseteq W_k$ according to the first $V_s$ containing each point gives
\[
\|P_BT\|\le\varepsilon_k\quad\text{for every finite } B\subseteq(N_k,\infty)\ \text{with}\ S_B\subseteq W_k .
\]
The finite set $S_{A'_k}$ does not meet the closed set $\bigcup_{j<k}\overline{U_j}\subseteq\bigcup_{j<k}W_j$. Choose an open $U_k\supseteq S_{A'_k}$ with $\overline{U_k}\subseteq W_k\setminus\bigcup_{j<k}\overline{U_j}$. Then the $U_k$ are pairwise disjoint.

Let $B_j$, $j<k$, be the set of $n\in A_k\setminus A'_k$ for which $j$ is the least index with $s_n\in W_j$. Since $B_j\subseteq(N_j,\infty)$ and $S_{B_j}\subseteq W_j$, we get $\|P_{B_j}T\|\le\varepsilon_j$, and
\[
\|P_{A'_k}Ty_k\|\ge\|P_{A_k}Ty_k\|-\sum_{j<k}\|P_{B_j}T\|\ge\delta-\sum_{j\ge1}\delta2^{-j-2}\ge\tfrac34\delta .
\]
The sets $A'_k$, the open sets $U_k$ and the elements $y_k$ contradict Lemma~\ref{l:hump} with $\varepsilon=\frac34\delta$.
\end{proof}

\section{Some remarks about the spectrum of compact disjointness preserving operators}

We now turn to spectral properties of compact disjointness preserving operators. In \cite{Ka} H.~Kamowitz proved the following.

\begin{theorem}
Let $K$ be a compact Hausdorff space and $T$ be a compact weighted composition operator on $C(K)$, $Tf=w(f\circ\varphi)$, where $w\in C(K)$ and $\varphi:K\to K$ is a continuous map. Let $\lambda\in\mathbb C\setminus\{0\}$. The following statements are equivalent.
\begin{enumerate}
\item[(1)] $\lambda\in\sigma(T)$.
\item[(2)] There is a $\varphi$-periodic $k\in K$ of the smallest period $p$ such that $\lambda^p=w(k)w(\varphi(k))\cdots w(\varphi^{p-1}(k))$.
\end{enumerate}
\end{theorem}

In the current section we will prove the following generalization of Theorem~3.1.

\begin{theorem}
Let $X$ be a Banach lattice, let $T:X\to X$ be a compact disjointness preserving operator, and let $\lambda\in\sigma(T)\setminus\{0\}$. Then there are an atom $F\in X'$ and a $p\in\mathbb N$ such that $(T')^pF=\lambda^pF$ and the elements $F,\dots,(T')^{p-1}F$ are pairwise disjoint atoms.
\end{theorem}

\begin{proof}
Recall that every atom of $X'$ is a scalar multiple of a real lattice homomorphism. Indeed, let $\phi\ge0$ be an atom of $X'$ and let $u,v\in X_+$ be disjoint. The functional $\phi_u(y)=\lim_n\phi(y\wedge nu)$, $y\in X_+$, extends to an element of $X'$ with $0\le\phi_u\le\phi$ (it is the component of $\phi$ corresponding to the band generated by $u$), so $\phi_u=t\phi$ for some $t\ge0$. Since $\phi_u(u)=\phi(u)$ and $\phi_u(v)=0$, we get $\phi(u)\phi(v)=0$. Consequently, if $G$ is an atom of $X'$ and $y,z\in X$ are disjoint, then $G(y)=0$ or $G(z)=0$.

By Theorem~\ref{t:AK}, $T=\sum_nF_n\otimes y_n$, where the $F_n$ are pairwise disjoint atoms of $X'$, the $y_n\in X$ are pairwise disjoint, and the series converges in operator norm. We may assume that $\|F_n\|=1$ and $y_n\ne0$ for all $n$. If there are infinitely many terms, then $\|y_n\|=\|F_n\otimes y_n\|\to0$. By the preceding paragraph, for every $n$ there is at most one $m$ with $F_n(y_m)\ne0$. If such an $m$ exists, we write $m=\theta(n)$ and $a_n=F_n(y_{\theta(n)})$; otherwise we put $a_n=0$. Since $T'\psi=\sum_m\psi(y_m)F_m$ for $\psi\in X'$, we have
\[
T'F_n=a_nF_{\theta(n)}\quad(\text{and }T'F_n=0\text{ if }a_n=0).
\]

Since $T$ is compact and $\lambda\ne0$, $\lambda$ is an eigenvalue of $T$. Let $x\ne0$, $Tx=\lambda x$, and put $b_n=F_n(x)/\lambda$. Then $x=\lambda^{-1}Tx=\sum_nb_ny_n$, so $b_n\ne0$ for some $n$, and $|b_n|\le\|x\|/|\lambda|$ for all $n$. Applying $F_n$ to $x=\sum_mb_my_m$ we obtain
\begin{equation}\label{e:cyc}
\lambda b_n=a_nb_{\theta(n)}\qquad(\text{the right-hand side is }0\text{ if }a_n=0).
\end{equation}
Fix $n_0$ with $b_{n_0}\ne0$. By \eqref{e:cyc}, $a_{n_0}\ne0$, so $n_1=\theta(n_0)$ is defined and $b_{n_1}=\lambda b_{n_0}/a_{n_0}\ne0$. Inductively we obtain indices $n_{k+1}=\theta(n_k)$ with $a_{n_k}\ne0$ and $|b_{n_{k+1}}|=|\lambda|\,|b_{n_k}|/|a_{n_k}|$. If the $n_k$ were pairwise distinct, there would be infinitely many terms and $n_k\to\infty$, so $|a_{n_k}|\le\|y_{n_{k+1}}\|\to0$. Then $|b_{n_{k+1}}|\ge2|b_{n_k}|$ for all large $k$, which contradicts the boundedness of $(b_n)$ (cf.\ \cite[Proposition~4]{Ka}). Hence the sequence $(n_k)$ is eventually periodic: there are an index $m=n_i$ and a smallest $p\in\mathbb N$ with $\theta^p(m)=m$, and the indices $m,\theta(m),\dots,\theta^{p-1}(m)$ are distinct. Iterating \eqref{e:cyc} along this cycle gives $\lambda^pb_m=\bigl(\prod_{l<p}a_{\theta^l(m)}\bigr)b_m$, so $\lambda^p=\prod_{l<p}a_{\theta^l(m)}$.

Put $F=F_m$. Then $(T')^iF=\bigl(\prod_{l<i}a_{\theta^l(m)}\bigr)F_{\theta^i(m)}$ for $0\le i\le p$. In particular $(T')^pF=\lambda^pF$, and $F,\dots,(T')^{p-1}F$ are nonzero multiples of the pairwise disjoint atoms $F_m,\dots,F_{\theta^{p-1}(m)}$.
\end{proof}

For $X=C(K)$ the atoms of $X'$ are the multiples of point masses, and Theorem~3.2 gives the periodic points of Theorem~3.1.

\begin{remark}
In the published version, Theorem~3.2 was stated for an arbitrary continuous disjointness preserving operator $T$ on a Banach lattice, under the assumptions that $\lambda\ne0$ is an isolated point of $\sigma(T)$, that $|\lambda|$ is isolated in $|\sigma(T)|$, and that $\dim\ker(\lambda I-T)<\infty$. The proof referred to \cite[Theorem~13.10]{AAK}. That theorem, which concerns Dedekind complete Banach lattices, describes when a circle $|\lambda|\mathbb T$ is not contained in $\sigma(T)$ in terms of a quotient of $X$; it does not provide atoms of $X'$. So the published argument is incomplete, and we do not know whether the conclusion of Theorem~3.2 holds for non-compact $T$ under these assumptions. Only the compact case is used below.
\end{remark}

\begin{corollary}
Let $X$ be a Banach $C(K)$-module, $T:X\to X$ be a compact disjointness preserving operator, and $\lambda\in\sigma(T)\setminus\{0\}$. Then there are an atom $F\in X'$ and a $p\in\mathbb N$ such that $(T')^pF=\lambda^pF$ and the elements $F,\dots,(T')^{p-1}F$ are pairwise disjoint atoms.
\end{corollary}

For $s\in K$ let $\Phi(s)=\{H\in X':\ H(fy)=f(s)H(y)\ \text{for all}\ f\in C(K),\ y\in X\}$. Every nonzero $H\in\Phi(s)$ satisfies $fH=f(s)H$ for the adjoint action of $C(K)$ on $X'$, so it is an atom in $X'$ \cite[Corollary~7.7.1]{AAK}. By Remark~2.2, nonzero elements of $\Phi(s)$ and $\Phi(t)$, $s\ne t$, are disjoint.

\begin{proof}
\emph{Step 1: fibre projections.} For $y\in X$ and $\psi\in X'$ let $\mu_{y,\psi}$ be the regular Borel measure on $K$ with $\int f\,d\mu_{y,\psi}=\psi(fy)$, $f\in C(K)$; $\|\mu_{y,\psi}\|\le\|\psi\|\,\|y\|$. Put $(\chi_s\psi)(y)=\mu_{y,\psi}(\{s\})$. Then $\chi_s\psi\in X'$ and $\|\chi_s\psi\|\le\|\psi\|$. Since $\mu_{fy,\psi}=f\mu_{y,\psi}$, we have $(\chi_s\psi)(fy)=f(s)(\chi_s\psi)(y)$, that is, $\chi_s\psi\in\Phi(s)$. If $H\in\Phi(s)$, then $\mu_{y,H}=H(y)\delta_s$.

\emph{Step 2: $T'$ maps fibres into fibres.} Let $H\in\Phi(s)$ and $y\in X$. By Theorem~1.7, $T$ maps the Banach lattice $X(y)$ into $X(Ty)$, and the restriction is a compact disjointness preserving operator between Banach lattices. By Theorem~1.4, $Tz=\sum_j\delta_j(z)u_j$ for $z\in X(y)$, with pairwise disjoint atoms $\delta_j$ of $X(y)'$ and pairwise disjoint $u_j\in X(Ty)$. As in the proof of Theorem~2.1, $\delta_j(fz)=f(t_j)\delta_j(z)$ for points $t_j\in K$. The restriction of $H$ to $X(Ty)$ is a scalar multiple of a real lattice homomorphism, because it is a fibre functional of the Banach lattice $X(Ty)$: on the dense sublattice $\{f\,Ty\}$ it is $f\,Ty\mapsto f(s)H(Ty)$, and $|f\,Ty|=|f|\,Ty$, since $g\mapsto g\,Ty$ is a lattice homomorphism of $C(K)$ onto a dense sublattice of $X(Ty)$ \cite{AAK}. So $H(u_j)\ne0$ for at most one $j$, and
\[
\mu_{y,T'H}=\delta_j(y)H(u_j)\,\delta_{t_j}\quad\text{or}\quad\mu_{y,T'H}=0 .
\]
Thus each $\mu_{y,T'H}$ is a multiple of a point mass. Since $y\mapsto\mu_{y,T'H}$ is linear, the point is the same for all $y$ with $\mu_{y,T'H}\ne0$. Otherwise $\mu_{y+y',T'H}$ would have two atoms. Hence $T'H\in\Phi(t)$ for some $t\in K$.

The point $t$ does not depend on $H\in\Phi(s)$. If $T'H_1\in\Phi(t_1)$ and $T'H_2\in\Phi(t_2)$ are nonzero with $t_1\ne t_2$, then $T'(H_1+H_2)\in\Phi(t)$ for some $t$. Applying $\chi_{t_1}$ and $\chi_{t_2}$ gives $t=t_1$ and $t=t_2$, which is impossible. So $T'\Phi(s)\subseteq\Phi(\tau(s))$ for a point $\tau(s)$ whenever $T'|_{\Phi(s)}\ne0$.

\emph{Step 3: the cyclic subspace.} Since $T$ is compact and $\lambda\ne0$, $\lambda$ is an eigenvalue. Let $x\ne0$ with $Tx=\lambda x$. Then $T$ maps $X(x)$ into $X(Tx)=X(x)$, and $T_x=T|_{X(x)}$ is a compact disjointness preserving operator on the Banach lattice $X(x)$ with eigenvalue $\lambda$. Theorem~3.2 applies to $T_x$ and gives an atom $G_0\in X(x)'$ and $p\in\mathbb N$ such that $(T_x')^pG_0=\lambda^pG_0$, and $G_i=(T_x')^iG_0$, $0\le i<p$, are pairwise disjoint atoms of $X(x)'$. As above, $G_i(fz)=f(s_i)G_i(z)$ for $z\in X(x)$, with points $s_i\in K$. The fibre functionals of $X(x)$ at a point are determined by their value at $x$, so they form a space of dimension at most one. Hence the points $s_0,\dots,s_{p-1}$ are distinct, and $G_i(x)\ne0$. Put $s_p=s_0$.

\emph{Step 4: lifting.} Let $\rho:X'\to X(x)'$ be the restriction map. Then $\rho T'=T_x'\rho$, and $\rho(\Phi(s_i))\subseteq\mathbb CG_i$. Let $\hat G\in X'$ be a Hahn--Banach extension of $G_0$, and $F_0=\chi_{s_0}\hat G\in\Phi(s_0)$. For $z\in X(x)$, $\mu_{z,\hat G}=G_0(z)\delta_{s_0}$, so $\rho F_0=G_0$. By induction on $i$, using Step~2, $(T')^iF_0\in\Phi(t)$ for some $t$, and $\rho(T')^iF_0=G_i$. Since $G_i(x)\ne0$ and $G_i(fx)=f(s_i)G_i(x)$, while also $G_i(fx)=f(t)G_i(x)$, we get $t=s_i$. Hence $T'\Phi(s_{i-1})\subseteq\Phi(s_i)$ for $1\le i\le p$, and $L=(T')^p|_{\Phi(s_0)}$ is a compact operator on the closed subspace $\Phi(s_0)$ of $X'$.

The functional $\ell(H)=(\rho H)(x)/G_0(x)$ on $\Phi(s_0)$ satisfies $\rho H=\ell(H)G_0$, $\ell(F_0)=1$, and
\[
\ell(LH)G_0=\rho(T')^pH=(T_x')^p\rho H=\ell(H)\lambda^pG_0 .
\]
Thus $L'\ell=\lambda^p\ell$, so $\lambda^p\ne0$ belongs to $\sigma(L')=\sigma(L)$. Since $L$ is compact, there is $F\in\Phi(s_0)$, $F\ne0$, with $(T')^pF=LF=\lambda^pF$.

\emph{Step 5.} The functionals $(T')^iF\in\Phi(s_i)$, $0\le i<p$, are nonzero, since $(T')^pF=\lambda^pF\ne0$. They are therefore atoms in $X'$ at the distinct points $s_0,\dots,s_{p-1}$, and they are pairwise disjoint by Remark~2.2.
\end{proof}

\begin{remark}
One should mention that A.W.~Wickstead (see \cite[Corollary~4]{W2}) obtained a complete description of the spectrum of a compact lattice homomorphism $T$ on a Banach lattice in the case when the conjugate operator $T'$ also preserves disjointness.
\end{remark}

\subsection*{Acknowledgement}
The revision of this paper was carried out with the help of Claude, an AI assistant developed by Anthropic. Claude found the counterexample to the implication (1)$\Rightarrow$(3) in Theorem~2.1 of the published version, the gaps in the proofs of Theorem~3.2 and Corollary~3.4, and the misstatements of Kamowitz's theorems (Theorems~1.3 and~3.1). It proposed the corrected statements, the proofs of (3)$\Rightarrow$(4) and (2)$\Rightarrow$(3), the proof of Theorem~3.2 for compact operators, and the completed proof of Corollary~3.4, and the revised text was checked by an independent instance of the same assistant. The first-named author has verified all arguments and takes full responsibility for the content of the revision.

\end{document}